\documentclass[12pt]{amsart}

\usepackage{amsmath}
\usepackage{amsthm}
\usepackage{latexsym}
\usepackage{amssymb}
\usepackage{mathrsfs}
\usepackage{tikz}
\DeclareMathOperator{\Char}{Char}

\DeclareMathOperator{\ad}{ad}

\DeclareMathOperator{\ssp}{sing\ supp}

\newtheorem{theorem}{Theorem}
\newtheorem{proposition}{Proposition}

\newtheorem{definition}{Definition}

\newtheorem{remark}{Remark}

\makeatletter
\@namedef{subjclassname@2020}{\textup{2020} Mathematics Subject Classification}
\makeatother
\def\jpopn#1#2{%
  \mathopen{%
    \setbox0=\hbox{$#1\langle$}%
    \setbox2=\hbox{%
            {\hbox{$#1\langle$}}%
            \kern -.6\wd0\box0%
    }%
            \box2%
  }%
}

\def\jpcls#1#2{%
  \mathclose{%
    \setbox0=\hbox{$#1\rangle$}%
    \setbox2=\hbox{%
            {\hbox{$#1\rangle$}}%
            \kern -.6\wd0\box0%
    }%
            \box2%
  }%
}

\def\tp{\ {}^{t}\kern -3pt}

\allowdisplaybreaks
\AtBeginDocument{
  \renewcommand{\setminus}{\mathbin{\backslash}}%
}
\def\uu#1{$\breve{\textrm{#1}}$}
\def\vv#1{$\check{\textrm{#1}}$}

\def\np#1{\ensuremath{\mathbf{NP}(\kern -1pt\it{#1})}}
\def\dnp#1{\ensuremath{\partial\mathbf{NP}(\kern -1pt\it{#1})}}

\renewcommand{\kappa}{\ensuremath{\varkappa}}
\renewcommand{\phi}{\ensuremath{\varphi}}
\renewcommand{\epsilon}{\ensuremath{\varepsilon}}

\def\R {{\mathbb{R}}}
\def\N {{\mathbb{N}}}

\def\Z {{\mathbb{Z}}}

\def\T {{\mathbb{T}}}
\begin{document}

\title[Globally Analytic hypoelliptic Sums of Squares]{On a Class of
Globally Analytic Hypoelliptic Sums of Squares}
\author{Antonio Bove}
\address{Dipartimento di Matematica, Universit\`a
di Bologna, Piazza di Porta San Donato 5, Bologna
Italy}
\email{bove@bo.infn.it}
\author{Gregorio Chinni}
\address{Dipartimento di Matematica, Universit\`a
di Bologna, Piazza di Porta San Donato 5, Bologna
Italy}
\email{gregorio.chinni3@unibo.it}
\date{\today}
\begin{abstract}
We consider sums of squares operators globally defined on the
torus. We show that if some assumptions are satisfied the operators
are globally analytic hypoelliptic. The purpose of the assumptions is
to rule out the existence of a Hamilton leaf on the characteristic
variety lying along the fiber of the cotangent bundle, i.e. the case
of the (global) M\'etivier operator. 
\end{abstract}
\maketitle
%
%

\section{Introduction}

The aim of this paper is to study a class of sums of squares operators
defined globally on a torus $ \T^{n} $. More precisely let $ P $ be a
sum of squares operator of the form
\begin{equation}
  \label{eq:sosq}
P (x, D) = \sum_{j=1}^{N} X_{j} (x, D)^{2} . 
\end{equation}
Here $ X_{j}(x, D) $ denotes a vector field with real analytic
coefficients defined on $ \T^{n} $, and we shall always assume that
H\"ormander condition is satisfied:
\begin{itemize}
\item[\textbf{(H)}]{}
  The vector fields and their iterated commutators generate a Lie
  algebra of the same dimension of the ambient space, i.e. $ n $. 
\end{itemize}
We are concerned with the regularity of the solutions to the equation
$ Pu = f $, where $ f \in C^{\omega}(\T^{n}) $. It is known that in
general the solutions $ u $ are not real analytic on $ \T^{n}
$. Actually let us consider the global M\'etivier operator
\begin{equation}
\label{eq:metop}
P_{M} = D_{x}^{2} + (\sin x)^{2} D_{y}^{2} + ( \sin y \  D_{y})^{2} .
\end{equation}
%
%
\begin{proposition}[see Treves, \cite{trevespienza}]
\label{prop:met}
The operator $ P_{M} $ in \eqref{eq:metop} is globally Gevrey 2
hypoelliptic on the torus $ \T^{2} $ and not better. 
\end{proposition}
We recall the definition of the Gevrey classes:
\begin{definition}
  \label{def:gevrey}
Let $ U $ be an open subset of $ \R^{n} $. 
The space $G^{s}(U)$, $s\geq 1$, the class of Gevrey functions of
order $ s $ in $U$, denotes the set of all $f \in C^{\infty}(U)$ 
such that for every compact set $K \Subset U$ there are two positive
constants $C_{K}$ and $A$ such that for every $\alpha \in
\mathbb{Z}^{n}_{+}$
$$
|D^{\alpha}f(x)| \leq A C^{|\alpha|}_{K} |\alpha|^{s|\alpha|}, \qquad \forall x\in K.
$$
$ G^{s}(\T^{n}) $, the class of the global Gevrey functions of order $
s $ in $ \T^{n} $, denotes the set of all $f \in C^{\infty}(T^{n})$
such that there exists a positive constant $ C $, for which
$$ 
|D^{\alpha}f(x)| \leq C^{|\alpha|+1} |\alpha|^{s|\alpha|}, \qquad
\forall \alpha \in \Z_{+}^{n}, x\in \T^{n}.
$$
\end{definition}
\begin{definition}
\label{def:hyp}
An operator $ P $ is said to be $G^{s}$-hypoelliptic, $s\geq 1$, 
in $U$, open subset of $\mathbb{R}^{n}$, if for any $U'$ open
subset of $U$, the conditions 
$u\in \mathscr{D}'\left(U\right)$ and $Pu \in G^{s}(U)$ imply that $u
\in G^{s}(U')$.

$ P $ is said to be globally $G^{s}$-hypoelliptic, $s\geq 1$, in $
\T^{n} $ if the conditions 
$u\in \mathscr{D}'\left(\T^{n}\right)$ and $Pu \in G^{s}(\T^{n})$
imply that $u \in G^{s}(\T^{n})$.
\end{definition}
\begin{proof}[Proof of Proposition \ref{prop:met}]
First of all we remark that the characteristic variety of $ P $ is $
\Char(P) = \{(0, 0; 0, \eta) \ |\ \eta \neq 0\} \cup \{(\pi, \pi; 0,
\eta) \ |\ \eta \neq 0\} $. 
Let $ U_{1} $ be a neighborhood of the origin in $ \T^{2} $ and $
U_{2} $ a neighborhood of the point $ (\pi, \pi) $ in $ \T^{2}
$. Moreover denote by $ V_{j} $, $ j = 1, 2 $, open neighborhoods of $
(0, 0) $, $ (\pi, \pi) $ respectively in $
\T^{2} $, such that $ V_{j} \Subset U_{j} $. Then $ \{ U_{1}, U_{2},
\T^{2}\setminus (\bar{V}_{1} \cup \bar{V}_{2}) \}$ is an open covering
of $ \T^{2} $.

Using
\cite{metivier81}, we know that any solution in $ U_{j} $ of the equation
$ P_{M}u_{j} = f \in C^{\omega}(U_{j}) $ belongs to $ G^{2}(U_{j}) $ and
is not better than that. Moreover $ u_{1}
\in C^{\omega}(U_{1} \setminus \{(0, 0)\})$  and $ u_{2}
\in C^{\omega}(U_{2} \setminus \{(\pi, \pi)\})$, since $ P_{M} $ is analytic
hypoelliptic on any open set not containing the origin and the points
$ (0, 0) $ and $ (\pi, \pi) $.

Since the cohomology groups with coefficients in the sheaf of real
analytic functions vanish in degree $ \geq 1 $ (see \cite{krantz},
\cite{grauert}, \cite{hcartan}),  we have that, on $ U_{1} \setminus
\bar{V}_{1} $, $ u_{1} = g_{1} - f_{1}
$, where $ g_{1} \in C^{\omega}(U_{1}) $, $ f_{1} \in
C^{\omega}(\T^{2} \setminus (\bar{V}_{1} \cup \bar{V}_{2}) $.
Analogously, on $ U_{2} \setminus \bar{V}_{2} $, we have $ u_{2} =
g_{2} - f_{1} $, where $ g_{2} \in C^{\omega}(U_{2}) $.

Define now $ v = u_{1} - g_{1} $ on $ U_{1} $, $ v = u_{2} - g_{2} $
on $ U_{2} $ and $ v = - f_{1} $ on $ \T^{2} \setminus(\bar{V}_{1}
\cup \bar{V}_{2}) $. We have that $ P_{M} v \in C^{\omega}(\T^{2}) $
and $ \ssp_{a} v \subset \{(0, 0)\} \cup \{(\pi, \pi)\} $ and is
nonempty.

This proves the statement.
\end{proof}
\begin{remark}
\label{rem:metgen}
An analogous proof can be made also for operators vanishing of higher
order on the characteristic variety, e.g.
\begin{equation}
\label{eq:metop2}
P_{1} = D_{x}^{2} + (\sin x)^{2(q-1)} D_{y}^{2} + ((\sin y)^{a} \  D_{y})^{2} ,
\end{equation}
where $ q > 2 $, $ a \geq 2 $. The only difference is that
M\'etivier's theorem \cite{metivier81} has to be replaced by that
proved in \cite{bm-metgen}.

We explicitly point out that both $ P_{M} $ and $ P_{1} $ have a
characteristic variety which is a non symplectic real analitic
manifold, where the Hamilton leaves lie along the $ \eta $ fibers of
the cotangent bundle.

On the other hand there are situations where, even though the
characteristic variety is a symplectic real analitic manifold, still
we have the occurrence of ``strata'' of M\'etivier type, i.e. whose
Hamilton leaves lie along the fibers of the cotangent bundle.
A class of models of this type is given by
\begin{equation}
\label{eq:metgen2}
P_{2} = D_{x}^{2} + (\sin x)^{2(q-1)} D_{y}^{2} + ((\sin x)^{p-1}
(\sin y)^{a} \  D_{y})^{2} , 
\end{equation}
where $ 1 < p < q $, $ a \geq 1 $. The analog of M\'etivier result
\cite{metivier81} is not known in general, however, if $ q = 2p $,
Chinni in \cite{chinni} has proved the optimality of the Gevrey
regularity $ s_{0} $, where 
$$ 
\frac{1}{s_{0}} = 1 - \frac{1}{2a} . 
$$
\end{remark}
It seems reasonable to surmise that the only globally non analytic
hypoelliptic operators are operators where the ``M\'etivier
situation'' occurs, meaning that there is a Hamilton leaf lying along
the fibers of the cotangent bundle.

Unfortunately this poses some major problem even in its formulation,
since we do not know a good way to stratify the characteristic
variety, let alone to define a Hamilton leaf, due to \cite{abm}, at
least in dimension greater or equal to 3.

Just in passing we mention that in special situations we actually are
able to define the strata and we refer to \cite{monom} for this.

\bigskip

In the present paper we consider a class of sums of squares operators
where we think that M\'etivier type situations do not occur.

More precisely let $ n $, $ m $ be two positive integers and $ P
$ be as in \eqref{eq:sosq}. We assume the following:
\begin{itemize}
\item[\textbf{(1)}]{}
The $ X_{j} $ are real analytic vector fields defined on the torus $
\T^{n+m} $. We denote the variable as $ (t, x) $ where $ t \in \T^{n}
$, $ x \in \T^{m} $.
\item[\textbf{(2)}]{}
Let $ n' < n $ and consider $ X_{j}(t, x, D_{t}, D_{x}) $ for $ 1 \leq
j \leq n'$. We assume that
\begin{equation}
\label{eq:nprime}
X_{j} = \sum_{i = 1}^{n'} a_{ji}(t') D_{t_{i}} ,
\end{equation}
where $ t = (t', t'') $ with $ t' \in \T^{n'} $, $ t'' \in \T^{n-n'}
$. Furthermore we assume that the vector fields $ X_{j} $, $ 1 \leq j
\leq n' $, are linearly independent for every $ t' \in \T^{n'} $.
\item[\textbf{(3)}]{}
Consider now $ X_{j} $ for $ n'+1 \leq j \leq N $. We assume that $ N
\geq n $ and that $ X_{j} $ has the form
\begin{equation}
\label{eq:Xj}
X_{j} = a_{j}(t') D_{t_{q(j)}} + \sum_{k=1}^{m} b_{jk}(t) D_{x_{k}} ,
\end{equation}
where $ a_{j} $, $ b_{jk} $ are real analytic funtions defined in $
\T^{n'} $, $ \T^{n} $ respectively and $ q $ is a surjective map from
$ \{ n'+1, \ldots, N\} $ onto $ \{n'+1, \ldots, n\} $.
Hence $ q^{-1}(\{j\}) $ is a partition of $ \{n'+1, \ldots, N\} $ with
non empty subsets.

Furthermore we assume that for each $ j = n'+1, \ldots, n $, there
exists $ \lambda_{j} \in  q^{-1}(\{j\}) $, such that
\begin{equation}
\label{eq:est1}
\sum_{r=n'+1}^{N} \sum_{k=1}^{m} | b_{rk}(t)| \leq C |a_{\ell}(t')| ,
\end{equation}
for every $ \ell \in \{\lambda_{j} \ | \ j \in \{n'+1, \ldots, n\} \}
$.

We also assume that
\begin{equation}
\label{eq:est2}
\sum_{r=n'+1}^{N} | a_{r}(t')| \leq C |a_{\ell}(t')| ,
\end{equation}
for every $ \ell \in \{\lambda_{j} \ | \ j \in \{n'+1, \ldots, n\} \}
$.
\item[\textbf{(4)}]{}
The vector fields $ X_{j} $, $ 1 \leq j \leq N $, satisfy H\"ormander
condition. 
\end{itemize}
\begin{remark}
\label{rem:2}
We could also consider the vector fields above in the case when $ n =
n' $ with no condition \eqref{eq:est1}. Then the corresponding
operator is in a subclass of that considered by Cordaro and Himonas in
\cite{ch94}.
\end{remark}
The vector fields described above can be used to produce the global
analog of some well known examples.

\bigskip
\noindent
\textbf{Example 1.}
Take $ n'=1 $, $ n=2 $, $ m=1 $. Then $ X_{1} = D_{t_{1}} $, $ X_{2} =
D_{t_{2}}$ and $ X_{3} = a(t_{1}) D_{x} $, where $ a $ denotes a non
identically vanishing real analytic function defined on $ \T^{1} $,
give the globally defined version of the (possibly generalized)
Baouendi-Goulaouic operator:
$$ 
D_{t_{1}}^{2} + D_{t_{2}}^{2} + a^{2}(t_{1}) D_{x}^{2} ,
$$
see also \cite{ch94}. We recall that the local version of the
Baouendi-Goulaouic operator is given by $ D_{t_{1}}^{2} +
D_{t_{2}}^{2} + t_{1}^{2k} D_{x}^{2} $, $ k \in \Z_{+} $. 

\par\noindent
\textbf{Example 2.}
Take $ n'=1 $, $ n=2 $, $ m=1 $. Then $ X_{1} = D_{t_{1}} $, $ X_{2} =
a(t_{1}) D_{t_{2}}$ and $ X_{3} = b(t_{1}) D_{x} $, where $ a $, $ b $
are non identically vanishing real analytic functions defined on $
\T^{1} $. Condition \eqref{eq:est1} becomes $ |b(t_{1})| \leq C
|a(t_{1})| $. The corresponding operator
$$ 
D_{t_{1}}^{2} + a^{2}(t_{1}) D_{t_{2}}^{2} + b^{2}(t_{1}) D_{x}^{2} ,
$$
is a globally defined version of
\begin{itemize}
\item[\textit{i)}]{}
the Ole\u \i nik--Radkevi\c c operator if $ b $ vanishes only where $
a $ vanishes. We recall that the local version of the Ole\u \i
nik--Radkevi\c c operator is given by $ D_{t_{1}}^{2} +
t_{1}^{2(p-1)} D_{t_{2}}^{2} + t_{1}^{2(q-1)} D_{x}^{2} $, $ p, q \in
\Z_{+} $, $ 1 < p \leq q $.
\item[\textit{ii)}]{}
the Baouendi-Goulaouic operator if $ b $ vanishes and $ a $ does not.
\item[\textit{iii)}]{}
an elliptic operator if neither $ a $ nor $ b $ vanish.
\end{itemize}
For a generalization of this example we refer to the paper
\cite{chinni14}. 

\par\noindent
\textbf{Example 3.}
Let $ n' = n = 2 $, $ m = 2 $. Then $ X_{1} = D_{t_{1}} $, $ X_{2} =
D_{t_{2}} $, $ X_{3} = a(t_{1}) D_{x_{1}} $, $ X_{4} = a(t_{1})
D_{x_{2}} $, $ X_{5} = b(t_{2}) D_{x_{1}} $, $ X_{6} = c(t_{2})
D_{x_{2}} $, $ a $, $ b $, $ c $ real analytic functions.

Consider the corresponding operator
\begin{equation}
  \label{eq:abm}
D_{t_{1}}^{2} + D_{t_{2}}^{2} + a^{2}(t_{1})\left(D_{x_{1}}^{2} +
  D_{x_{2}}^{2}\right) + b^{2}(t_{2}) D_{x_{1}}^{2} + c^{2}(t_{2})
D_{x_{2}}^{2} .
\end{equation}
Assume that $ a $, $ b $, $ c $ vanish at the origin of order $ r-1 $,
$ p-1 $ and $ q-1 $ respectively, with $ r < p < q $. Then the
operator above is the global version of the operator
$$ 
D_{t_{1}}^{2} + D_{t_{2}}^{2} + t_{1}^{2(r-1)}\left(D_{x_{1}}^{2} +
  D_{x_{2}}^{2}\right) + t_{2}^{2(p-1)} D_{x_{1}}^{2} + t_{2}^{2(q-1)}
D_{x_{2}}^{2},
$$
which has been proved to violate Treves conjecture in \cite{abm}. On
the other hand Chinni in \cite{chinni2} has proved that the above
globally defined operator is analytic hypoelliptic. The operator in
\eqref{eq:abm} is also in the class studied by Cordaro and Himonas,
\cite{ch94}. 

If we choose $ p < q < r $, then the corresponding operator in
\eqref{eq:abm} is again globally analytic hypoelliptic and satisfies
Treves conjecture.

\par\noindent
\textbf{Example 4.}
Let $ n' = 1 $, $ n = 2 $ and $ m = 2 $. Then $ X_{1} = D_{t_{1}} $, $ X_{2} =
a(t_{1}) D_{t_{2}} $, $ X_{3} = b(t_{1}) D_{x_{1}} $, $ X_{4} = b(t_{1})
D_{x_{2}} $, $ X_{5} = c(t_{1}, t_{2}) D_{x_{1}} $, $ X_{6} = d(t_{1},
t_{2}) D_{x_{2}} $, $ a $, $ b $, $ c $ and $ d $ are real analytic
functions.

Consider the corresponding operator
\begin{equation}
  \label{eq:bm}
D_{t_{1}}^{2} + a^{2}(t_{1}) D_{t_{2}}^{2} + b^{2}(t_{1})\left(D_{x_{1}}^{2} +
  D_{x_{2}}^{2}\right) + c^{2}(t_{1}, t_{2}) D_{x_{1}}^{2} +
d^{2}(t_{1}, t_{2}) D_{x_{2}}^{2} .
\end{equation}
Assume that $ a $ does not vanish and also that the operator in
\eqref{eq:bm} is not elliptic. Then condition \eqref{eq:est1} implies
that the operator in \eqref{eq:bm} is a slight generalization of that
in \eqref{eq:abm}. In particular it belongs to the class studied in
\cite{ch94}. 

Assume then that $ a $ vanishes at the origin. Then $ a(t_{1}) =
t_{1}^{\ell} \tilde{a}(t_{1}) $, where $ \ell \in \N $, $ \tilde{a}(0)
\neq 0$. Hence condition \eqref{eq:est1} implies that $ b, c, d =
\mathscr{O}(t_{1}^{\ell}) $.

Assume that $ b = \mathscr{O}(t_{1}^{\ell+r-1}) $ for a certain $ r
\in \N $, $ r > 1 $. Moreover assume that $ t_{1}^{-\ell} c(t) =
\mathscr{O}(t_{2}^{p-1}) $, and that $ t_{1}^{-\ell} d(t) =
\mathscr{O}(t_{2}^{q-1}) $, for certain $ p $, $ q \in \N $, $ 1 < r <
p < q$.

Then the operator in \eqref{eq:bm} is the global analog of the
operator
$$ 
D_{t_{1}}^{2} + t_{1}^{2\ell} D_{t_{2}}^{2} + t_{1}^{2(\ell+r-1)}\left(D_{x_{1}}^{2} +
  D_{x_{2}}^{2}\right) + t_{1}^{2\ell}\Big( t_{2}^{2(p-1)} D_{x_{1}}^{2} + t_{2}^{2(q-1)}
D_{x_{2}}^{2} \Big) ;
$$
it has a symplectic characteristic real analytic manifold having a
symplectic stratification according to Treves. It has been proved to
violate Treves conjecture in \cite{bm}.

This operator does not belong to the class studied in \cite{ch94}.

\par\noindent
\textbf{Example 5.}
Let $ n' = 1 $, $ n = 3 $ and $ m = 1 $. Then $ X_{1} = D_{t_{1}} $, $
X_{2} = a_{2}(t_{1}) D_{t_{2}} $, $ X_{3} = a_{3}(t_{1}) D_{t_{3}} +
b(t) D_{x} $, where $ a_{2} $, $ a_{3} $, $ b $ are real analytic
functions in $ \T^{3} $.

Assume that $ a_{2} = \mathscr{O}(t_{1}^{p-1}) $ for $ t_{1} \to 0 $,
then condition \eqref{eq:est2} implies that $ a_{3} =
\mathscr{O}(t_{1}^{p-1}) $ for $ t_{1} \to 0 $, and condition
\eqref{eq:est1} inplies that $ b = \mathscr{O}(t_{1}^{p-1}) $ for $
t_{1} \to 0 $, i.e. $ b(t) = \mathscr{O}(t_{1}^{p-1}) \tilde{b}(t) $.

For this example, when condition \eqref{eq:est2} is not satisfied, we
do not know a general analytic hypoellipticity result.

\bigskip

We state now the main result of the paper
\begin{theorem}
\label{th:1}
Let $ P $ be a sum of squares operator as in \eqref{eq:sosq} with real
analytic coefficients globally defined on the torus $ \T^{n+m}
$. Assume that conditions \textbf{(2), (3), (4)} above are
satisfied. Then $ P $ is globally analytic hypoelliptic.
\end{theorem}
We point out explicitly that even though $ P $ is globally analytic
hypoelliptic, it is not in general locally analytic hypoelliptic, as
we can see from the above examples.

\section{Proof of Theorem \ref{th:1}}

\setcounter{equation}{0}
\setcounter{theorem}{0}
\setcounter{proposition}{0}
\setcounter{lemma}{0}
\setcounter{corollary}{0}
\setcounter{definition}{0}
\setcounter{remark}{0}

In order to prove the theorem we use the maximal hypoellipticity
global $ L^{2} $ estimate for $ P $.
\begin{equation}
\label{eq:max}
\sum_{j=1}^{N} \| X_{j} u \|_{0}^{2} \leq C \Big( \langle Pu, u
\rangle + \| u\|_{0}^{2} \Big) ,
\end{equation}
where $ u \in C^{\infty}(\T^{n+m}) $. We explicitly remark that the
subelliptic term $ \| u \|_{\epsilon}^{2} $---Sobolev norm of order $
\epsilon $, which is present in the 
estimates proved by H\"ormander and Rothschild and Stein in
\cite{hormander67}, \cite{rothschild-stein76}, is not
needed, since we cannot use it to prove analytic regularity, but only
to prove a Gevrey regularity. 

Since the characteristic variety of $ P $ is contained in
$$
\{(t', t'', x; 0, \tau'', \xi) \ | \ |\tau''| + |\xi| > 0 \} ,
$$
we may use $ t'' $- and $ x $-derivatives to establish real
analyticity. 

Actually we want to bind the quantity
$$ 
\Big ( \sum_{j=n'+1}^{n} D_{t_{j}}^{p} + \sum_{k=1}^{m} D_{x_{k}}^{p}
\Big) u ,
$$
where $ p $ is a large integer and $ u $ is a smooth solution of $ P u
= f $, with $ f \in C^{\omega}(\T^{n+m}) $ and $ P $ is given by
\begin{multline}
\label{eq:P}
P(t, D_{t}, D_{x}) = \sum_{j=1}^{n'} \left(\sum_{i = 1}^{n'}
  a_{ji}(t') D_{t_{i}} \right)^{2}
\\
+ \sum_{j=n'+1}^{N}
\left( a_{j}(t') D_{t_{q(j)}} + \sum_{k=1}^{m} b_{jk}(t) D_{x_{k}}
\right)^{2} .
\end{multline}

\subsection{The $ x $-derivatives.}
Let $ k \in \{1, \ldots, m\} $. We want to estimate $ \| X
D_{x_{k}}^{p} u \|_{0} $, where $ X $ denotes one of the vector fields
$ X_{j} $, $ j=1, \ldots, N $.

To treat the error term on the right hand side of \eqref{eq:max} we
use the subelliptic estimate with a generic subelliptic term. We know
that there exists a positive number, $ \epsilon $, such that the Lie
algebra is generated by brackets of length $ \leq \epsilon^{-1} $ so
that we have
\begin{equation}
\label{eq:subell}
\| u \|_{\epsilon}^{2} +
\sum_{j=1}^{N} \| X_{j} u \|_{0}^{2} \leq C \Big( \langle Pu, u
\rangle + \| u\|_{0}^{2} \Big) .
\end{equation}
Let us start by considering $ \| D_{x_{k}}^{p} u \|_{0} $, i.e. the
error term on the right hand side of \eqref{eq:subell} where $ u $ has
been replaced by $ D_{x_{k}}^{p} u $.

We start off by showing that the error term $\| D_{x_{k}}^{p} u
\|_{0}^{2}$ can actually be absorbed in the l.h.s. of \eqref{eq:subell}. To this end,
denote by $ \chi $ a smooth cutoff function such that $
\chi(t) = 1 $ if $ |t| \geq 2 $ and $ \chi(t) = 0 $ if $ |t|
\leq 1 $. It turns out that $\chi(p^{-1} \langle D \rangle)\in OPS^0$,
where $ \langle D \rangle = (1 + |D|^{2})^{\frac{1}{2}} $. 
(see Def. \ref{def:symb} in  Appendix) and then
\begin{equation}
\label{eq:2nd_cutoff}
\| D_{x_{k}}^{p} u \|_{0} \leq \| (1 -
\chi(p^{-1}\langle D \rangle) ) D_{x_{k}}^{p} u \|_{0}
+ \|\chi(p^{-1} \langle D \rangle)  D_{x_{k}}^{p} u \|_{0} .
\end{equation}
The first summand can be easily estimated, using Proposition
\ref{pseudo}, because the support of the cutoff $ 1 - \chi $ is
contained in $ \{\xi \in \Z^{m} \ | \ |\xi| \leq 2 p\} $.

Whence we obtain
$$
\| (1 - \chi(p^{-1}\langle D \rangle) ) D_{x_{k}}^{p} u \|_{0}
\leq C^{p+1} p^p,
$$
which is an analytic growth estimate.

Thus we are left with the estimate of the second summand in the
r.h.s. of \eqref{eq:2nd_cutoff}. We have that
$$
\|\chi(p^{-1} \langle D \rangle)  D_{x_{k}}^{p} u \|_{0} = p^{-\epsilon}
\| p^{\epsilon}\chi(p^{-1} \langle D \rangle) \langle D
\rangle^{-\epsilon} \langle D \rangle^{\epsilon} D_{x_{k}}^{p} u \|_{0} .
$$
Due to the support of the cutoff $\chi$, we see that
$$
\sigma\big(p^{\epsilon}\chi(p^{-1} \langle D \rangle) \langle D
\rangle^{-\epsilon} \big) = p^{\epsilon }\chi(p^{-1}\langle \xi
\rangle) \langle \xi \rangle^{-\epsilon}
\in S^{0}
$$
with the $S^{0}$-semi-norms uniformly bounded with respect to $p$;
thus the $ L^{2} $ continuity theorem \ref{pseudo} it follows that
$$
\|p^{\epsilon}\chi(p^{-1} \langle D \rangle) \langle D
\rangle^{-\epsilon} \|_{\mathcal{L}(L^2,L^2)}\leq C,
$$
$C$ being a positive constant independent of $p$. Summarizing we
obtained the following inequality
\begin{equation}
\label{eq:error}
\| D_{x_{k}}^{p} u \|_{0} \leq  C^{p+1} p^p + C p^{-\epsilon} \|
D_{x_{k}}^{p} u \|_{\epsilon} .
\end{equation}
Hence, plugging $ D_{x_{k}}^{p}u $ into \eqref{eq:subell}, because of
\eqref{eq:error}, we obtain
\begin{equation}
\label{eq:Dxp-1}
\| D_{x_{k}}^{p} u \|_{\epsilon}^{2} +
\sum_{j=1}^{N} \| X_{j} D_{x_{k}}^{p} u \|_{0}^{2} \leq C \Big(
\langle P D_{x_{k}}^{p} u, D_{x_{k}}^{p} u \rangle + C^{2p} p^{2p} \Big),
\end{equation}
with possibly a larger constant $ C $. Here the term $C p^{-\epsilon} \|
D_{x_{k}}^{p} u \|_{\epsilon} $ has been absorbed on the left hand
side of \eqref{eq:Dxp-1}.

Due to conditions \textbf{(1) -- (3)}, $ D_{x_{k}} $ commutes with $ P
$, so that
\begin{equation}
\label{eq:Dxp-2}
\| D_{x_{k}}^{p} u \|_{\epsilon}^{2} +
\sum_{j=1}^{N} \| X_{j} D_{x_{k}}^{p} u \|_{0}^{2} \leq C \Big(
\| D_{x_{k}}^{p} P u\|_{0}^{2} + \| D_{x_{k}}^{p} u \|_{0}^{2} + C^{2p} p^{2p} \Big).
\end{equation}
Hence, applying \eqref{eq:error} to the next to last term above and
keeping into account that $ Pu $ is an analytic function, we obtain
\begin{equation}
\label{eq:Dxp-3}
\| D_{x_{k}}^{p} u \|_{\epsilon} \leq C^{p+1} p^{p},
\end{equation}
for every $ k = 1, \ldots, m $, i.e. an analytic growth rate with
respect to the variable $ x $.

\subsection{The $ t'' $-derivatives.}
Next we consider $ D_{t_{j}}^{q} $, $ j \in \{n'+1, \ldots, n\} $ and
$ q $ is a large integer. As above from \eqref{eq:subell} we deduce
\begin{equation}
\label{eq:Dtq-1}
\| D_{t_{j}}^{q} u \|_{\epsilon}^{2} + \sum_{r=1}^{N} \| X_{r}
D_{t_{j}}^{q} u \|_{0}^{2} \leq C \Big( \langle P D_{t_{j}}^{q} u,
D_{t_{j}}^{q} u \rangle +  \| D_{t_{j}}^{q} u \|_{0}^{2} \Big) .
\end{equation}
Applying \eqref{eq:error} to the second term on the right hand side of
the above inequality we obtain, with a possibly larger constant,
\begin{equation}
\label{eq:Dtq-1}
\| D_{t_{j}}^{q} u \|_{\epsilon}^{2} + \sum_{r=1}^{N} \| X_{r}
D_{t_{j}}^{q} u \|_{0}^{2} \leq C \Big( \langle P D_{t_{j}}^{q} u,
D_{t_{j}}^{q} u \rangle +  C^{2q} q^{2q} \Big) .
\end{equation}
Consider now $ \langle P D_{t_{j}}^{q} u, D_{t_{j}}^{q} u \rangle
$. We have
\begin{multline*}
\langle P D_{t_{j}}^{q} u, D_{t_{j}}^{q} u \rangle = \sum_{r=1}^{N}
\langle X_{r}^{2} D_{t_{j}}^{q} u,  D_{t_{j}}^{q} u \rangle
\\
=
\sum_{r=1}^{N} \Big(
\langle  D_{t_{j}}^{q} X_{r}^{2} u,  D_{t_{j}}^{q} u \rangle
+
\langle [ X_{r}^{2} , D_{t_{j}}^{q} ] u,  D_{t_{j}}^{q} u \rangle
\Big) .
\end{multline*}
Now
$$ 
[ X_{r}^{2} , D_{t_{j}}^{q} ] = 2 X_{r} [ X_{r} , D_{t_{j}}^{q} ] - [
[  D_{t_{j}}^{q} , X_{r} ] , X_{r} ].
$$
Let us examine first $ \langle X_{r} [ X_{r} , D_{t_{j}}^{q} ] u,
D_{t_{j}}^{q} u \rangle $, for $ r= 1, \ldots, N $. If $ r \in \{1,
\ldots, n'\} $ then the commutator is zero, so that we have to
consider the case $ r \in \{n'+1, \ldots, N\} $. Since in this case $
X_{r} $ is a self adjoint vector field, we have
$$ 
\langle X_{r} [ X_{r} , D_{t_{j}}^{q} ] u, D_{t_{j}}^{q} u \rangle =
\langle  [ X_{r} , D_{t_{j}}^{q} ] u, X_{r} D_{t_{j}}^{q} u \rangle.
$$
By Cauchy-Schwartz we get
$$ 
\left| \langle X_{r} [ X_{r} , D_{t_{j}}^{q} ] u, D_{t_{j}}^{q} u
  \rangle \right| \leq \delta \| X_{r} D_{t_{j}}^{q} u \|_{0}^{2} +
C_{\delta} \| [ X_{r} , D_{t_{j}}^{q} ] u \|_{0}^{2} .
$$
Choosing $ \delta $ in a convenient way allows us to absorb the first
norm on the right hand side above on the left hand side of
\eqref{eq:Dtq-1}.

Thus
$$ 
[ D_{t_{j}}^{q}, X_{r} ] = \sum_{k=1}^{m} [ D_{t_{j}}^{q} , b_{rk}(t)
] D_{x_{k}}
=
\sum_{k=1}^{m} \sum_{\ell=1}^{q} \binom{q}{\ell} (\ad
D_{t_{j}})^{\ell}(b_{rk}) D_{t_{j}}^{q-\ell} D_{x_{k}} .
$$
Hence
\begin{equation}
  \label{eq:comm}
\| [ X_{r} , D_{t_{j}}^{q} ] u \|_{0} \leq \sum_{k=1}^{m}
\sum_{\ell=1}^{q} \binom{q}{\ell} \| (D_{t_{j}}^{\ell}b_{rk})
D_{t_{j}}^{q-\ell} D_{x_{k}} u \|_{0} .
\end{equation}
By condition \eqref{eq:est1} each coefficient $ |b_{rk}| \leq C
|a_{\ell}| $, for every $ \ell \in \{ \lambda_{j} \ | \ j \in \{n'+1,
\ldots, n\}\} $ and for every $ k \in \{1, \ldots, m\} $, $ r \in
\{n'+1, \ldots, N\} $.

If the $ a_{\ell} $ do not vanish on $ \T^{n'} $ for $ \ell \in \{
\lambda_{j} \ | \ j \in \{n'+1, \ldots, n\}\} $, we do nothing.

Assume that $ a_{\ell}^{-1}(0) \neq \varnothing $. For this we need a
partition of unity in $ \T^{n'} $. Let $ U_{h} $, $ h = 1, \ldots, M $,
denote an open covering of $ \T^{n'} $ and let $ \phi_{h}(t') \in
C_{0}^{\infty}(U_{h}) $, $ \sum_{h=1}^{M} \phi_{h}(t') = 1 $, be a
partition of unity subordinated to the covering $ U_{h} $.  

We are going to estimate $ \| \phi_{h} [ X_{r} , D_{t_{j}}^{q} ] u
\|_{0}  $, $ h = 1, \ldots, M $.

We want to discuss the restriction of $ b_{rk} $ to $ U_{h} $. For the
sake of simplicity we argue for $ U_{h} $ when the origin belongs to $
U_{h} $ and we assume that $ a_{\ell}^{-1}(0) \cap U_{h} \neq
\varnothing $. We may always assume that $ a_{\ell}(0) = 0 $, since
the estimate we obtain is independent of $ h $.

Without loss of generality we may also assume that $ a_{\ell}(t_{1},
0) $ has an isolated zero of multiplicity $ p $ at zero. Then by the
Weierstra\ss\ preparation theorem, see e.g. \cite{treves-book}, we may
write
$$ 
a_{\ell}(t') = e_{\ell}(t') \frak{p}_{\ell}(t_{1}, \tilde{t}) ,
$$
where
$$ 
\frak{p}_{\ell}(t_{1}, \tilde{t}) = t_{1}^{p} + \sum_{j=1}^{p} a_{\ell
  j}(\tilde{t}) t_{1}^{p-j},
$$
where $ \tilde{t} = (t_{2}, \ldots, t_{n'}) $, $ e_{\ell}(t') \in
C^{\omega}(U_{h}) $ nowhere vanishing and $ a_{\ell j}(\tilde{t}) $
are real analytic functions vanishing at $ \tilde{t} = 0 $.

Consider now $ b_{rk}(t_{1}, 0, t'') $. It vanishes when $ t_{1} = 0
$, for every $ t'' \in \T^{n''} $ due to our assumptions. We may
always perform a small linear change of the variables $ t' $ in such a
way that both $ a_{\ell}(t_{1}, 0) $ and  $ b_{rk}(t_{1}, 0, t'') $
have an isolated zero at $ t_{1} = 0 $, even though of different
multiplicities.

Then by the Weierstra\ss\ preparation theorem we may write
$$ 
b_{rk}(t', t'') = \epsilon_{rk}(t', t'') \frak{q}_{rk}(t_{1}, \tilde{t}, t''), 
$$
where
$$ 
\frak{q}_{rk}(t_{1}, \tilde{t}, t'') = t_{1}^{q} + \sum_{j=1}^{q} b_{kr
  j}(\tilde{t}, t'') t_{1}^{q-j} .
$$
Our assumption that $ | b_{rk} | \leq C | a_{\ell} | $ imply that $ |
\frak{q}_{rk} | \leq C_{1} | \frak{p}_{\ell}| $. This implies that $ q
\geq p$ and that $ \frak{p}_{\ell} $ divides $ \frak{q}_{rk} $, i.e.
\begin{equation}
\label{eq:brk}
b_{rk}(t) = \epsilon_{rk}(t) \frak{p}_{\ell}(t_{1}, \tilde{t})
\frak{q}'(t_{1}, \tilde{t}, t''),
\end{equation}
for every $ r \in \{n'+1, \ldots, N\} $, $ k \in \{1, \ldots, m\} $
and $ \ell \in \{\lambda_{j} \ | \ j \in \{n'+1, \ldots, n\}\} $. Here
$ \frak{q}' $ denotes a Weierstra\ss\ polynomial of degree $ q-p $.

We note explicitly that the dependence on $ t'' $ of the functions $
b_{rk} $ is confined to the unity $ \epsilon_{rk} $ and the
coefficients of $ \frak{q}' $.

As a consequence of the above argument we have that
\begin{equation}
\label{eq:Dbrk}
D_{t_{j}}^{\ell_{1}} b_{rk}(t) = D_{t_{j}}^{\ell_{1}}
(\epsilon_{rk}(t) \frak{q}'(t_{1}, \tilde{t}, t''))
\frak{p}_{\ell}(t_{1}, \tilde{t}). 
\end{equation}
Going back to \eqref{eq:comm}, we rewrite \eqref{eq:Dbrk} as
\begin{multline}
  \label{eq:Dellb}
D_{t_{j}}^{\ell} b_{rk}(t) = D_{t_{j}}^{\ell}
(\epsilon_{rk}(t) \frak{q}'(t_{1}, \tilde{t}, t''))
\frak{p}_{\lambda_{j}}(t_{1}, \tilde{t})
\\
= D_{t_{j}}^{\ell}
(\epsilon_{rk}(t) \frak{q}'(t_{1}, \tilde{t}, t''))
e_{\lambda_{j}}(t')^{-1} a_{\lambda_{j}}(t') .
\end{multline}
Then
\begin{multline*}
\| \phi_{h} [ X_{r} , D_{t_{j}}^{q} ] u \|_{0}
\\
\leq \sum_{k=1}^{m}
\sum_{\ell=1}^{q} \binom{q}{\ell} \| \phi_{h} D_{t_{j}}^{\ell}
(\epsilon_{rk}(t) \frak{q}'(t_{1}, \tilde{t}, t''))
e_{\lambda_{j}}(t')^{-1} a_{\lambda_{j}}(t') 
D_{t_{j}}^{q-\ell} D_{x_{k}} u \|_{0} .
\end{multline*}
There are two cases: the first is $ \ell = q $ and the second is $
\ell < q $. In the first case we obtain
\begin{multline}
\label{eq:ell=q}
\sum_{k=1}^{m} \| \phi_{h} D_{t_{j}}^{q}
(\epsilon_{rk}(t) \frak{q}'(t_{1}, \tilde{t}, t''))
e_{\lambda_{j}}(t')^{-1} a_{\lambda_{j}}(t') D_{x_{k}} u \|_{0}
\\
\leq
\sum_{k=1}^{m} C_{1}^{q+1} q! \| D_{x_{k}} u \|_{0} \leq C_{2}^{q+1}
q! \| u \|_{0} .
\end{multline}
Consider the second case
\begin{multline}
\label{eq:ell<q}
\sum_{k=1}^{m}
\sum_{\ell=1}^{q-1} \binom{q}{\ell} \| \phi_{h} D_{t_{j}}^{\ell}
(\epsilon_{rk}(t) \frak{q}'(t_{1}, \tilde{t}, t''))
e_{\lambda_{j}}(t')^{-1} a_{\lambda_{j}}(t') 
D_{t_{j}}^{q-\ell} D_{x_{k}} u \|_{0}
\\
\leq
\sum_{k=1}^{m}
\sum_{\ell=1}^{q-1} C_{2}^{\ell+1} q^{\ell}  \| a_{\lambda_{j}}(t') D_{t_{j}} 
D_{t_{j}}^{q-\ell-1} D_{x_{k}} u \|_{0}
\\
\leq
\sum_{k=1}^{m}
\sum_{\ell=1}^{q-1} C_{2}^{\ell+1} q^{\ell}  \| X_{\lambda_{j}} 
D_{t_{j}}^{q-\ell-1} D_{x_{k}} u \|_{0}
\\
+
\sum_{k=1}^{m}
\sum_{\ell=1}^{q-1} \sum_{k_{1}=1}^{m}  C_{2}^{\ell+1} q^{\ell}  \|
b_{\lambda_{j} k_{1}}  
D_{t_{j}}^{q-\ell-1} D_{x_{k_{1}}} D_{x_{k}} u \|_{0} .
\end{multline}
The first term is ready for an induction and hence we have to consider
the last term above: by assumption \eqref{eq:est1} we may write
\begin{multline*}
\sum_{k=1}^{m}
\sum_{\ell=1}^{q-1} \sum_{k_{1}=1}^{m}  C_{2}^{\ell+1} q^{\ell}  \|
b_{\lambda_{j} k_{1}}  
D_{t_{j}}^{q-\ell-1} D_{x_{k_{1}}} D_{x_{k}} u \|_{0}
\\
\leq
C \sum_{k=1}^{m}\sum_{k_{1}=1}^{m} C_{2}^{q} q^{q-1}  \| D_{x_{k_{1}}}
D_{x_{k}} u \|_{0} 
\\
+
C
\sum_{k=1}^{m}
\sum_{\ell=1}^{q-2} \sum_{k_{1}=1}^{m}  C_{2}^{\ell+1} q^{\ell}  \|
a_{\lambda_{j}} D_{t_{j}} 
D_{t_{j}}^{q-\ell-2} D_{x_{k_{1}}} D_{x_{k}} u \|_{0}
\\
\leq
C \sum_{k=1}^{m}\sum_{k_{1}=1}^{m} C_{2}^{q} q^{q-1}  \| D_{x_{k_{1}}}
D_{x_{k}} u \|_{0}
\\
+
C
\sum_{k=1}^{m}
\sum_{\ell=1}^{q-2} \sum_{k_{1}=1}^{m}  C_{2}^{\ell+1} q^{\ell}  \|
X_{\lambda_{j}} D_{t_{j}}^{q-\ell-2} D_{x_{k_{1}}} D_{x_{k}} u \|_{0}
\\
+
C
\sum_{k=1}^{m}
\sum_{\ell=1}^{q-2} \sum_{k_{1}=1}^{m} \sum_{k_{2}=1}^{m} C_{2}^{\ell+1} q^{\ell}  \|
b_{\lambda_{j} k_{2}}  D_{t_{j}}^{q-\ell-2} D_{x_{k_{2}}}
D_{x_{k_{1}}} D_{x_{k}} u \|_{0} .
\end{multline*}
Again the first term yields an analytic growth rate, the second is
ready for an induction. For the last term
above we may iterate $ q-1 $ times the procedure obtaining
\begin{multline*}
\sum_{k=1}^{m}
\sum_{\ell=1}^{q-1} \binom{q}{\ell} \| \phi_{h} D_{t_{j}}^{\ell}
(\epsilon_{rk}(t) \frak{q}'(t_{1}, \tilde{t}, t''))
e_{\lambda_{j}}(t')^{-1} a_{\lambda_{j}}(t') 
D_{t_{j}}^{q-\ell} D_{x_{k}} u \|_{0}
\\
\leq
\sum_{i=1}^{q-1} \sum_{k=1}^{m} \sum_{k_{1}=1}^{m}  \cdots
\sum_{k_{i-1}=1}^{m} C_{3}^{q+1} q^{q-i+1} \|D_{x_{k}} D_{x_{k_{1}}}
\cdots D_{x_{k_{i-1}}} u \|_{0}
\\
+
\sum_{k=1}^{m} \sum_{k_{1}=1}^{m}  \cdots \sum_{k_{q-2}=1}^{m}
C_{3}^{q+1} q \| X_{\lambda_{j}} D_{x_{k}} D_{x_{k_{1}}}
\cdots D_{x_{k_{q-2}}} u \|_{0}
\\
+
\sum_{k=1}^{m} \sum_{k_{1}=1}^{m}  \cdots \sum_{k_{q-1}=1}^{m}
C_{3}^{q+1} q \| D_{x_{k}} D_{x_{k_{1}}} \cdots D_{x_{k_{q-1}}} u
\|_{0} .
\end{multline*}
Consider now the second summation on the right hand side above. Using
\eqref{eq:Dxp-2}, \eqref{eq:Dxp-3}, and remarking that in those
inequalities one may swap the derivative $ D_{x_{k}}^{p} $ with $
D_{x}^{\alpha} $, where $ \alpha $ is a multliindex with $ |\alpha|=p
$, we have that
\begin{equation}
\label{eq:s1}
\sum_{k=1}^{m} \sum_{k_{1}=1}^{m}  \cdots \sum_{k_{q-2}=1}^{m}
C_{3}^{q+1} q \| X_{\lambda_{j}} D_{x_{k}} D_{x_{k_{1}}}
\cdots D_{x_{k_{q-2}}} u \|_{0}
\leq
C_{4}^{q+1} q^{q} .
\end{equation}
The first sum and the third can be rewritten as
$$ 
\sum_{i=1}^{q} \sum_{k=1}^{m} \sum_{k_{1}=1}^{m}  \cdots
\sum_{k_{i-1}=1}^{m} C_{3}^{q+1} q^{q-i+1} \|D_{x_{k}} D_{x_{k_{1}}}
\cdots D_{x_{k_{i-1}}} u \|_{0} .
$$
Now, by \eqref{eq:error}, we have, for a multiindex $ \alpha $,
$$ 
\| D_{x}^{\alpha} u \|_{0} \leq  C^{|\alpha|+1} |\alpha|^{|\alpha|} + C |\alpha|^{-\epsilon} \|
D_{x}^{\alpha} u \|_{\epsilon} .
$$
Hence, using \eqref{eq:Dxp-3}, we get
\begin{equation}
  \label{eq:s2}
\sum_{i=1}^{q} \sum_{k=1}^{m} \sum_{k_{1}=1}^{m}  \cdots
\sum_{k_{i-1}=1}^{m} C_{3}^{q+1} q^{q-i+1} \|D_{x_{k}} D_{x_{k_{1}}}
\cdots D_{x_{k_{i-1}}} u \|_{0} \leq C_{4}^{q+1} q^{q} ,
\end{equation}
for a suitable positive constant $ C_{4} $.

Summing over $ h $ the above estimate implies that
\begin{equation}
\label{eq:s3}
\| [ X_{r} , D_{t_{j}}^{q} ] u \|_{0} \leq C_{4}^{q+1} q^{q} ,
\end{equation}
with a slightly larger constant.

Consider next the term, with $ r \in \{n'+1, \ldots, N\} $, $ j =
n'+1, \ldots, n $,
$$ 
\langle [ [  D_{t_{j}}^{q} , X_{r} ] , X_{r} ] u, D_{t_{j}}^{q} u
\rangle .
$$
We have
\begin{multline*}
[ [  D_{t_{j}}^{q} , X_{r} ], X_{r} ]
=
\sum_{k=1}^{m} \sum_{\ell=1}^{q} \binom{q}{\ell} [ (\ad
D_{t_{j}})^{\ell}(b_{rk}) D_{t_{j}}^{q-\ell}  , X_{r} ] D_{x_{k}}
\\
=
\sum_{k=1}^{m} \sum_{\ell=1}^{q} \binom{q}{\ell} [ 
(D_{t_{j}}^{\ell}b_{rk}) , X_{r} ] D_{t_{j}}^{q-\ell} D_{x_{k}}
\\
+
\sum_{k=1}^{m} \sum_{\ell=1}^{q} \sum_{k_{1}=1}^{m}
\sum_{\ell_{1}=1}^{q-\ell} \binom{q}{\ell} \binom{q-\ell}{\ell_{1}}
(D_{t_{j}}^{\ell}b_{rk}) (\ad D_{t_{j}})^{\ell_{1}}(b_{rk_{1}})
D_{t_{j}}^{q-\ell-\ell_{1}} D_{x_{k}} D_{x_{k_{1}}}
\\
=
- \sum_{k=1}^{m} \sum_{\ell=1}^{q} \binom{q}{\ell} 
a_{r}(t') ( D_{t_{q(r)}} D_{t_{j}}^{\ell}b_{rk})  D_{t_{j}}^{q-\ell}
D_{x_{k}}
\\
+
\sum_{k=1}^{m} \sum_{\ell=1}^{q} \sum_{k_{1}=1}^{m}
\sum_{\ell_{1}=1}^{q-\ell} \binom{q}{\ell} \binom{q-\ell}{\ell_{1}}
(D_{t_{j}}^{\ell}b_{rk}) (D_{t_{j}}^{\ell_{1}} b_{rk_{1}})
D_{t_{j}}^{q-\ell-\ell_{1}} D_{x_{k}} D_{x_{k_{1}}}
\\
=
J_{1} + J_{2}.
\end{multline*}
Let us examine $ J_{1} $ first. We have
\begin{multline*}
\langle J_{1} u, D_{t_{j}}^{q} u \rangle =
- \sum_{k=1}^{m} \sum_{\ell=1}^{q} \binom{q}{\ell}
\langle ( D_{t_{q(r)}} D_{t_{j}}^{\ell}b_{rk})  D_{t_{j}}^{q-\ell}
D_{x_{k}} u , a_{r}(t') D_{t_{j}}^{q} u \rangle
\\
=
- \sum_{k=1}^{m} \langle ( D_{t_{q(r)}} D_{t_{j}}^{q}b_{rk})
D_{x_{k}} u , a_{r}(t') D_{t_{j}}^{q} u \rangle
\\
- \sum_{k=1}^{m} \sum_{\ell=1}^{q-1} \binom{q}{\ell}
\langle ( D_{t_{q(r)}} D_{t_{j}}^{\ell}b_{rk})  D_{t_{j}}^{q-\ell -1}
D_{x_{k}} u , a_{r}(t') D_{t_{j}}^{q+1} u \rangle
\\
+ \sum_{k=1}^{m} \sum_{\ell=1}^{q-1} \binom{q}{\ell}
\langle ( D_{t_{q(r)}} D_{t_{j}}^{\ell+1}b_{rk})  D_{t_{j}}^{q-\ell -1}
D_{x_{k}} u , a_{r}(t') D_{t_{j}}^{q} u \rangle
\\
= J_{11} + J_{12} + J_{13} .
\end{multline*}
Consider $ J_{11} $. We have that
$$ 
|J_{11}| \leq C_{0}^{q+1} q! \sup_{k} \|D_{x_{k}} u \|_{0}
\|D_{t_{j}}^{q} u\|_{0} ,
$$
which gives an analytic growth rate by using an analog of
\eqref{eq:error}. Consider then $ J_{12} $. For the left hand side
factor of the scalar product we apply assumption \eqref{eq:est1} and
proceed analogously to \eqref{eq:ell<q}. For the right hand side
factor of the scalar product we apply assumption \eqref{eq:est2}.  We
obtain
\begin{multline}
  \label{eq:J12}
| J_{12} | \leq
\sum_{k=1}^{m} \sum_{\ell=1}^{q-1} \binom{q}{\ell} C_{b}^{\ell+2}
(\ell+1)!  \| X_{\lambda_{j}}  D_{t_{j}}^{q-\ell -2} D_{x_{k}} u
\|_{0} \|X_{\lambda_{j}} D_{t_{j}}^{q} u\|_{0}
\\
+
\sum_{k=1}^{m} \sum_{k_{1}=1}^{m}  \sum_{\ell=1}^{q-1} \binom{q}{\ell} C_{b}^{\ell+2}
(\ell+1)! \| b_{\lambda_{j} k_{1}}  D_{t_{j}}^{q-\ell -2} D_{x_{k}} D_{x_{k_{1}}} u
\|_{0} \|X_{\lambda_{j}} D_{t_{j}}^{q} u\|_{0}
\\
+
\sum_{k=1}^{m} \sum_{k_{1}=1}^{m}  \sum_{\ell=1}^{q-1} \binom{q}{\ell} C_{b}^{\ell+2}
(\ell+1)!  \| X_{\lambda_{j}}  D_{t_{j}}^{q-\ell -2} D_{x_{k}} u
\|_{0} \| b_{\lambda_{j} k_{1}}   D_{t_{j}}^{q} D_{x_{k_{1}}} u\|_{0}
\\
+
\sum_{k=1}^{m} \sum_{k_{1}=1}^{m}  \sum_{k'_{1}=1}^{m} \sum_{\ell=1}^{q-1} \binom{q}{\ell} C_{b}^{\ell+2}
(\ell+1)!  \|  b_{\lambda_{j} k'_{1}}  D_{t_{j}}^{q-\ell -2} D_{x_{k}}
D_{x_{k'_{1}}}u \|_{0}
\\
\cdot
\| b_{\lambda_{j} k_{1}}   D_{t_{j}}^{q}
D_{x_{k_{1}}} u\|_{0} .
\end{multline}
As for the first summand we observe that
$$ 
\binom{q}{\ell} (\ell+1)! \leq q^{\ell+1} ,
$$
so that the norm $ q^{\ell+1} \| X_{\lambda_{j}}  D_{t_{j}}^{q-\ell
  -2} D_{x_{k}} u \|_{0} $ yields an analytic growth rate, since the
norm $ \|X_{\lambda_{j}} D_{t_{j}}^{q} u\|_{0} $ can be absorbed on
the left hand side of the a priori estimate.

Let us consider the other summands in the above inequality. Observe
that the norms appearing in the sums are of the same type as those on
the right hand side of \eqref{eq:ell<q}. Hence arguing in the same way
we can conclude as we did for the case of the simple commutator (see
the argument preceding \eqref{eq:s2}.)

Consider then $ J_{13} $. It is a lower order term due to the fact
that one derivative landed onto $ b_{rk} $. Thus its treatment is
completely analogous to that of $ J_{12} $, but simpler once
\eqref{eq:error} is used to absorb on the left hand side $ q^{-\epsilon} \|
D_{t_{j}}^{q} u \|_{\epsilon} $.

Finally we have to examine $ J_{2} $:
\begin{multline*}
  \langle J_{2}u, D_{t_{j}}^{q} u \rangle
  =
\sum_{k=1}^{m} \sum_{\ell=1}^{q-1} \sum_{k_{1}=1}^{m}
\sum_{\ell_{1}=1}^{q-\ell} \binom{q}{\ell} \binom{q-\ell}{\ell_{1}}
\\
\cdot
\langle (D_{t_{j}}^{\ell}b_{rk}) (D_{t_{j}}^{\ell_{1}} b_{rk_{1}})
D_{t_{j}}^{q-\ell-\ell_{1}} D_{x_{k}} D_{x_{k_{1}}} u,  D_{t_{j}}^{q}
u \rangle + \left(C_{0}^{q+1} q!\right)^{2} + \delta \|D_{t_{j}}^{q} u
\|_{\epsilon}^{2},
\end{multline*}
where $ \delta $ is a small constant.

Next we are going to forget about the last two summands in the above
relation.

We start off by bringing a $ x $-derivative to the right hand side of
the scalar product and $ \ell_{1} $ $ t_{j} $-derivatives to the left,
so that
\begin{multline*}
  \langle J_{2}u, D_{t_{j}}^{q} u \rangle
  =
\sum_{k=1}^{m} \sum_{\ell=1}^{q-1} \sum_{k_{1}=1}^{m}
\sum_{\ell_{1}=1}^{q-\ell} \binom{q}{\ell} \binom{q-\ell}{\ell_{1}}
\\
\cdot
\langle (D_{t_{j}}^{\ell}b_{rk}) D_{t_{j}}^{q-\ell-\ell_{1}} D_{x_{k}}
u, (D_{t_{j}}^{\ell_{1}} b_{rk_{1}}) D_{t_{j}}^{q} D_{x_{k_{1}}} u
\rangle
\\
=
\sum_{k=1}^{m} \sum_{\ell=1}^{q-1} \sum_{k_{1}=1}^{m}
\sum_{\ell_{1}=1}^{q-\ell} \sum_{s=0}^{\ell_{1}}
\sum_{\sigma=0}^{\ell_{1}-s} (-1)^{s} 
\binom{q}{\ell} \binom{q-\ell}{\ell_{1}} \binom{\ell_{1}}{s}
\binom{\ell_{1}-s}{\sigma} 
\\
\cdot
\langle (D_{t_{j}}^{\ell+\sigma}b_{rk}) D_{t_{j}}^{q-\ell-s-\sigma} D_{x_{k}}
u, (D_{t_{j}}^{\ell_{1}+s} b_{rk_{1}}) D_{t_{j}}^{q-\ell_{1}} D_{x_{k_{1}}} u
\rangle ,
\end{multline*}
where the adjoint Newton binomial formula has been used.
Hence
\begin{multline*}
| \langle J_{2}u, D_{t_{j}}^{q} u \rangle |
\leq
\\
\sum_{k=1}^{m} \sum_{\ell=1}^{q-1} \sum_{k_{1}=1}^{m}
\sum_{\ell_{1}=1}^{q-\ell} \sum_{s=0}^{\ell_{1}}
\sum_{\sigma=0}^{\ell_{1}-s}
\frac{q!}{\ell!} \frac{1}{(q-\ell-\ell_{1})!} \frac{1}{s!}
\frac{1}{\sigma! (\ell_{1} - s - \sigma)!}
\\
\cdot
\| (D_{t_{j}}^{\ell+\sigma}b_{rk}) D_{t_{j}}^{q-\ell-s-\sigma} D_{x_{k}}
u \|_{0} \cdot
\| (D_{t_{j}}^{\ell_{1}+s} b_{rk_{1}}) D_{t_{j}}^{q-\ell_{1}}
D_{x_{k_{1}}} u \|_{0} .
\end{multline*}
By \eqref{eq:Dellb} we obtain
\begin{multline*}
| \langle J_{2}u, D_{t_{j}}^{q} u \rangle |
\leq
\\
\sum_{k=1}^{m} \sum_{\ell=1}^{q-1} \sum_{k_{1}=1}^{m}
\sum_{\ell_{1}=1}^{q-\ell} \sum_{s=0}^{\ell_{1}}
\sum_{\sigma=0}^{\ell_{1}-s} C_{b}^{\ell+\ell_{1}+s+\sigma+1}
\frac{q!}{\ell!} \frac{1}{(q-\ell-\ell_{1})!} \frac{1}{s!}
\frac{1}{\sigma! (\ell_{1} - s - \sigma)!}
\\
\cdot
(\ell+\sigma)!
(\ell_{1}+s)! 
\| a_{\lambda_{j}}(t') D_{t_{j}}^{q-\ell-s-\sigma} D_{x_{k}}
u \|_{0} \cdot
\| a_{\lambda_{j}}(t') D_{t_{j}}^{q-\ell_{1}}
D_{x_{k_{1}}} u \|_{0}
\\
\leq
\sum_{k=1}^{m} \sum_{\ell=1}^{q-1} \sum_{k_{1}=1}^{m}
\sum_{\ell_{1}=1}^{q-\ell} \sum_{s=0}^{\ell_{1}}
\sum_{\sigma=0}^{\ell_{1}-s} C_{1}^{\ell+\ell_{1}+s+\sigma+1}
q^{\ell+s+\sigma}
\\
\cdot
\| a_{\lambda_{j}}(t') D_{t_{j}} D_{t_{j}}^{q-\ell-s-\sigma-1} D_{x_{k}}
u \|_{0} \cdot q^{\ell_{1}}
\| a_{\lambda_{j}}(t') D_{t_{j}}  D_{t_{j}}^{q-\ell_{1}-1}
D_{x_{k_{1}}} u \|_{0} .
\end{multline*}
Then we may argue for each factor above as in \eqref{eq:ell<q} to get
analytic growth rate.

This ends the estimate of the left hand side of \eqref{eq:Dtq-1} for $
j \in \{n'+1, \ldots, n\} $. This implies that
\begin{equation}
\label{eq:fin}
| \Delta_{(t, x)}^{q} u | \leq C^{2q+1} q!^{2} , \text{ on } \T^{n+m} ,
\end{equation}
which proves our statement.

\section{Further thoughts on Example 5}

\setcounter{equation}{0}
\setcounter{theorem}{0}
\setcounter{proposition}{0}
\setcounter{lemma}{0}
\setcounter{corollary}{0}
\setcounter{definition}{0}
\setcounter{remark}{0}

Let us consider Example 5 without assuming condition \eqref{eq:est2}:
\begin{equation}
\label{eq:5}
D_{t_{1}}^{2} + a_{2}(t_{1})^{2} D_{t_{2}}^{2} + (a_{3}(t_{1})
D_{t_{3}} + b(t) D_{x})^{2} ,
\end{equation}
where
$$ 
a_{2}(t_{1}) = \mathscr{O}(t_{1}^{p-1}) \text{ for } t_{1} \to 0,
$$
and assuming that
$$ 
b(t) = \mathscr{O}(t_{1}^{k}) \tilde{b}(t) \text{ for } t_{1} \to 0,
\text{ where } k \geq 2(p-1) .
$$
The proof of the global analytic hypoellipticity of \eqref{eq:5}
differs from the above proof only in the estimate of the double
commutator term. 

More precisely the estimate of the right hand side factor of the
scalar product in \eqref{eq:J12} can be obtained by using the above
assumption on the function $ b $ instead that condition
\eqref{eq:est2}.

\setcounter{section}{0}
\renewcommand\thesection{\Alph{section}}

\section{Appendix}

\setcounter{equation}{0}
\setcounter{theorem}{0}
\setcounter{proposition}{0}
\setcounter{lemma}{0}
\setcounter{corollary}{0}
\setcounter{definition}{0}

For the sake of completeness we recall here some well-known facts
about pseudodifferential operators on the torus $ \T^{n} $.
\begin{definition}
\label{def:symb}
For any $m\in\R$, we denote by $ S^m(\T^{n})$
the set of all the functions $p(x,\xi)\in C^\infty(\T^{n}\times \Z^{n})$ such that for every multi-index $\alpha$, there exits 
a positive constant $C_{\alpha}$ for which 
$$
|\partial_x^\alpha p(x,\xi)|\leq C_{\alpha} \langle \xi \rangle^{m},
$$
where $\langle \xi \rangle=(1+|\xi|^2)^{\frac{1}{2}}$.

We denote by $OPS^m$ the class of the corresponding pseudodifferential
operators $P=p(x,D)$ whose action on smooth functions on the torus is
defined as
$$ 
p(x, D) u (x) = \sum_{\xi \in \Z^{n}} p(x, \xi) \hat{u}_{\xi} e^{i x
  \cdot \xi} ,
$$
where $ u(x) = \sum_{\xi \in \Z^{n}} \hat{u}_{\xi} e^{i x \cdot \xi} $
is the Fourier series of $ u $. 
\end{definition}
It is trivial to see that the symbol class $S^m(\T^{n})$ equipped with the semi-norms
\begin{equation*}\label{seminorms}
|p|^{(m)}_\ell=\max_{|\alpha|\leq \ell} \sup_{x \in \T^{n}} \{|\partial_x^\alpha p(x,\xi)|
\langle \xi \rangle^{-m}\} ,\quad \ell\in\N,
\end{equation*}
is a Fr\'echet space.

We need the $L^2$-continuity of the pseudodifferential operators in 
the above classes.
We state below a formulation of such a theorem for pseudodifferential operators of order $ \sigma $.
\begin{proposition}[Chinni and Cordaro, \cite{cc2016}]
  \label{pseudo}
Let $Q$ be a pseudodifferential operator on $\T^n$ of order
$\sigma\in\R$. Then there is a constant $M>0$ such that, for
every $k\in\Z_+$,
$$
\| Q u\|_{k} \leq M |q|^{(\sigma)}_{k+[|\sigma|]+n+1} \left(\|
  u\|_{k+\sigma} + \| u \|_{\sigma}\right)  , \quad 
u\in C^\infty(\T^n). 
$$
Here we denoted by $ [|\sigma|] $ the integral part of $ |\sigma|
$ plus 1, if $ |\sigma| $ is not an integer, and $ |\sigma| $ if $
|\sigma| $ is an integer.
\end{proposition}
\begin{proof}
Let $\tilde q(x,\eta)=e_{-\eta}(x)Q(e_{\eta})(x)$ be the discrete
symbol of $Q$. Here $ e_{\eta}(x) = e^{ix\cdot \eta} $. 
By the definition of a pseudodifferential operator of order $\sigma$,
there is a constant $C>0$ such that
$$
|D_x^\alpha \tilde q(x,\eta)| \leq C_{\alpha} (1+|\eta|)^\sigma, \quad x \in \T^n,  \alpha\in \Z^n_+.
$$ 
Set $q(\xi, \eta) = \langle Q(e_{\eta}), \, e_{\xi} \rangle_{0}$. Then the following inequality holds:
\begin{equation}\label{exp-ch}
|q(\xi,\eta)| \leq C_{N} (1+|\xi-\eta|)^{-N} (1+|\eta|)^\sigma, \quad \xi,\eta\in \Z^n,
\end{equation}
for every $ N \in \N $ and
for some constant $C_{N} > 0$. Moreover we can write
$$
\widehat{(Qu)}(\xi) = \sum_{\eta \in \Z^n} q(\xi,\eta)\hat u(\eta), \quad \xi\in\Z^n,\, u \in C^\infty(\T^N).
$$
Now we estimate
\begin{align*}
\| &Q u  \|_{k}
= \left( \sum_{\xi \in \Z^n} (1 + |\xi|)^{2k} |\widehat{Qu}(\xi)|^{2} \right)^{1/2}\\
&
= \left( \sum_{\xi \in \Z^{n}}
\left( \sum_{\eta\in \mathbb{Z}^{n}}(1+|\xi|)^{k}|q(\xi,\eta)||\widehat u(\eta)|\right)^{\!\!2}
\right)^{1/2}\\
&
\leq
\left( \sum_{\xi \in \mathbb{Z}^{n}}
\left( \sum_{\ell =0}^{k}\binom{k}{\ell}\sum_{\eta \in \mathbb{Z}^{n}}|\xi -\eta|^{\ell}
|q (\xi,\eta)| (1+ |\eta|)^{k-\ell}|\widehat{u}(\eta)|\right)^{\!\!\!2}\right)^{1/2}.
\end{align*}
Applying the Minkowsky's inequality we can estimate the right hand side of the above inequality:
\begin{align*}
\| &Q u  \|_{k}
\leq 
 \sum_{\ell =0}^{k}\binom{k}{\ell}
 \left( \sum_{\xi \in \mathbb{Z}^{n}}
 \left( \sum_{\eta \in \mathbb{Z}^{n}} |\xi-\eta|^{\ell}
 |q(\xi,\eta)| (1+ |\eta|)^{k-\ell}|\widehat{u}(\eta)|\right)^{\!\!\!2}\right)^{1/2}\\
 &
\leq
\sum_{\ell =0}^{k}\binom{k}{\ell}
 \left( \sum_{\xi \in \mathbb{Z}^{n}}
  \left(\sum_{\eta \in \mathbb{Z}^{n}} |\xi-\eta|^{2\ell}|q(\xi,\eta)|\right)
\right.
\\
&  
\left.  
\left(\sum_{\eta \in \mathbb{Z}^{n}} |q(\xi,\eta)| (1+ |\eta|)^{2(k-\ell)}
|\widehat{u}(\eta)|^{2}\right)\right)^{1/2} = I.
\end{align*}
Applying (\ref{exp-ch}) and Peetre inequality 
$$(1+ |\eta|)^{\sigma} \leq (1 + | \xi - \eta|)^{|\sigma|} (1+ | \xi
|)^{\sigma}, \quad \xi,\eta\in \Z^n,
$$
we have
\begin{align*}
I
&
\leq C_{k}
\sum_{\ell =0}^{k} \binom{k}{\ell} 
\left(\sum_{\xi \in \mathbb{Z}^{n}}\left(\sum_{\eta\in \Z^n} (1 +
    |\eta|)^{\sigma} (1 + |\xi - \eta|)^{-(n+1)} \right)
    \right.
  \\
  &
  \cdot  
  \left.
\left(\sum_{\eta \in \mathbb{Z}^{n}} |q(\xi,\eta)| (1+ |\eta|)^{2(k-\ell)}
|\widehat{u}(\eta)|^{2}\right)\right)^{1/2}\\
&
\leq
C_{k} \sum_{\ell =0}^{k} \binom{k}{\ell}
\left(\sum_{\xi \in \mathbb{Z}^{n}}(1 + |\xi|)^{\sigma} 
\sum_{\eta \in \mathbb{Z}^{n}} |q(\xi,\eta)| (1+ |\eta|)^{2(k-\ell)}
|\widehat{u}(\eta)|^{2}\right)^{1/2}\\
&
\leq C_{k} \sum_{\ell =0}^{k} \binom{k}{\ell}
\left(\sum_{\xi,\eta \in \mathbb{Z}^{n}}  (1+ |\eta|)^{2(k-\ell)+2\sigma}
(1+|\xi-\eta|)^{|\sigma| - [|\sigma|] -1} |\widehat{u}(\eta)|^{2}\right)^{1/2}\\
&
\leq
\tilde{C}_{k} \left( \| u\|_{k+ \sigma} + \| u \|_{\sigma} \right) ,
\end{align*}
where $ C_{k} $, $ \tilde{C}_{k} $ are suitable positive constants.

We explicitly point out that $ C_{k} \leq M
|q|^{(\sigma)}_{k+[|\sigma|] +n + 1} $, for a positive constant $ M
$.
\end{proof}


\begin{thebibliography}{100}
%
%
%
%
\bibitem{abm}
P. Albano, A. Bove and M. Mughetti,
{\it Analytic Hypoellipticity for Sums of Squares and the Treves
  Conjecture\/},
 J. Funct. Anal. {\bf 274} (2018), no. 10, 2725--2753.
%
\bibitem{bg72}
M. S. Baouendi and Ch. Goulaouic,
{\it Nonanalytic-hypoellipticity for some degenerate op-
erators\/}, Bull. A. M. S., {\bf 78} (1972), 483-486.
%
%
%
%
%
\bibitem{bm-j}
A. Bove and M. Mughetti,
{\it Analytic hypoellipticity for sums of squares in the presence of
  symplectic non Treves strata\/},
J. Inst. Math. Jussieu {\bf 19}(6) (2020), 1877--1888.
%
\bibitem{bm}
A. Bove and M. Mughetti,
{\it Analytic Hypoellipticity for Sums of Squares and the Treves
  Conjecture, II},
Analysis and PDE {\bf 10} (7) (2017), 1613--1635.
%
%
%
\bibitem{monom}
A. Bove and M. Mughetti,
{\it Gevrey Regularity for a Class of Sums of Squares of Monomial
  Vector Fields\/},
Advances in Math. {\bf 373}(2020) 107323, 35pp.
%
\bibitem{bm-surv}
A. Bove and M. Mughetti,
{\it Analytic regularity for solutions to sums of squares: an
  assessment\/},
Complex Analysis and its synergies {\bf 6}(2020), no. 2.
{\tt https://doi.org/10.1007/s40627-020-00055-8}.
%
\bibitem{bm-metgen}
A. Bove and M. Mughetti,
{\it Optimal Gevrey
  Regularity for Certain Sums of Squares in Two Variables\/},
preprint, 2021.
%
%
%
%
\bibitem{btreves}
A. Bove and F. Treves,
{\it On the Gevrey hypo-ellipticity of sums of squares of vector
  fields\/},  Ann. Inst. Fourier (Grenoble) {\bf 54}(2004),
1443-1475.
%
%
%
\bibitem{hcartan}
H. Cartan,
{\it Vari\'et\'es analytiques r\'eelles et vari\'et\'es analytiques
  complexes\/},
Bulletin de la Soci\'et\'e Math\'ematique de France {\bf 85} (1957),
77--99.
%
\bibitem{chinni14}
G. Chinni,
{\it Gevrey regularity for a generalization of the Ole\u \i
  nik--Radkevi\c c operator\/},
J. Math. Anal. Appl. {\bf 415} (2014), 948--962. 
%
\bibitem{cc2016}
G. Chinni and P.~D.~Cordaro,
{\it On global analytic and Gevrey hypoellipticity on the torus and
  the M\'etivier inequality\/},
Communications in P.D.E. {\bf 42}(1) (2017), 121--141.
%
\bibitem{chinni}
G. Chinni,
{\it On the sharp Gevrey regularity for a generalization of the
  M\'etivier operator\/},
preprint, 2021.
%
%
\bibitem{chinni2}
G. Chinni,
{\it (Semi)-global analytic hypoellipticity for a class of “sums of
  squares” which fail to be  locally analytic hypoelliptic\/},
Proc. A.M.S., to appear, \texttt{DOI:
  https://doi.org/10.1090/proc/14464}. 
%
%
%
%
\bibitem{ch94}
D. Cordaro and A. A. Himonas,
{\it Global analytic hypoellipticity for a class of degenerate
  elliptic operators on the torus\/},
Math. Res. Letters {\bf 1} (1994), 501--510.
%
%
%
%
%
%
%
%
%
\bibitem{grauert}
H. Grauert,
{\it On Levi's Problem and the Imbedding of Real-Analytic
  Manifolds\/},
Annals of Mathematics {\bf 68} (1958), 460--472. 
%
%
%
%
%
%
%
%
%
\bibitem{hormander67}
L. H\"ormander,
{\it Hypoelliptic second order differential equations\/},
Acta Math. {\bf 119} (1967), 147--171.
%
\bibitem{horm71}
L. H\"ormander, {\it Uniqueness Theorems and Wave Front Sets for
  Solutions of Linear Differential Equations with Analytic
  Coefficients\/},
Communications Pure Appl. Math. {\bf 24} (1971), 671--704.
%
\bibitem{horm1}
L. H\"ormander, {\it The Analysis of Partial Differential Operators,
  I\/}, Springer Verlag, 1985.
%
\bibitem{horm3}
L. H\"ormander, {\it The Analysis of Partial Differential Operators,
  III\/}, Springer Verlag, 1985.
%
%
%
%
%
%
\bibitem{krantz}
S. G. Krantz,
{\it Function theory of several complex variables\/}, 2nd ed., AMS
Chelsea Publishing, 2001, Providence, R.I.
%
\bibitem{metivier81}
G. M\'etivier,
{\it Non-hypoellipticit\'e Analytique pour $ D_{x}^{2}
  + (x^{2} + y^{2}) D_{y}^{2} $\/},
C. R. Acad. Sci. Paris S\'er. I Math. {\bf 292} (1981), no. 7, 401--404.
%
%
%
%
%
%
%
%
\bibitem{o72}
O. A. Ole\uu \i nik,
{\it On the analyticity of solutions of partial differential equations
  and systems\/},
Colloque International CNRS sur les \'Equations
aux D\'eriv\'ees Partielles Lin\'eaires (Univ. Paris-
Sud, Orsay, 1972), 272--285. Ast\'erisque, 2 et 3. Societ\'e
Math\'ematique de France, Paris, 1973.
%
%
\bibitem{ol-rad-73}
O. A. Ole\uu \i nik and E. V. Radkevi\vv c,
{\it The analyticity of the solutions of linear partial differential
  equations\/},
(Russian) Mat. Sb. (N.S.), {\bf 90(132)} (1973), 592--606.
%
%
%
%
\bibitem{rothschild-stein76}
  L. Rothschild and E. M. Stein,
  {\it Hypoelliptic differential operators and nilpotent groups\/},
  Acta Math. {\bf  137} (3 - 4) (1976), 247--320.
%
\bibitem{serre51}
J-P. Serre,
{\it Applications de la th\'eorie g\'en\'erale \`a divers probl\`emes
  globaux\/},
S\'eminaire Henri Cartan, 20, {\bf 4} (1951-1952).
%
%
%
%
%
%
%
%
%
%
\bibitem{Treves}
F. Treves,
{\it Symplectic geometry and analytic hypo-ellipticity, in
  Differential equations\/}, La Pietra 1996 (Florence),
Proc. Sympos. Pure Math. {\bf 65}, Amer. Math. Soc., Providence, RI,
1999, 201-219.
%
%
\bibitem{trevespienza}
F. Treves, {\it On the analyticity of solutions of sums of
  squares of vector fields\/},  Phase space analysis of partial
differential equations, Bove, Colombini, Del Santo ed.'s, 315-329,
Progr. Nonlinear Differential Equations Appl., {\bf 69},
Birkh\"auser Boston, Boston, MA, 2006.
%
\bibitem{treves-book}
F. Treves,
{\it Aspects of Analytic PDE\/}, book to appear.
%
%
%
\end{thebibliography}
\end{document}